\documentclass[11pt,letterpaper]{amsart}

\usepackage{enumerate}
\usepackage{amscd}
\usepackage{amsmath}
\usepackage{amssymb}
\usepackage{amsthm}    
\usepackage{latexsym}
\usepackage{color} 
\usepackage{graphicx}
\usepackage{url}

\numberwithin{equation}{section}

\theoremstyle{plain}
\newtheorem{theorem}{Theorem}[section]
\newtheorem{lemma}[theorem]{Lemma} 
\newtheorem{corollary}[theorem]{Corollary} 

\theoremstyle{definition}
\newtheorem{definition}[theorem]{Definition} 
\newtheorem{question}[theorem]{Question} 
\newtheorem{example}[theorem]{Example} 

\theoremstyle{remark}
\newtheorem{remark}[theorem]{Remark} 

\providecommand{\norm}[1]{\lVert#1\rVert}
\DeclareMathOperator{\diag}{diag}
\DeclareMathOperator{\rank}{rank}

\begin{document}

\title[Simplices with fixed volumes of codimension 2 faces]{Simplices with fixed volumes of codimension 2 faces in a continuous deformation}

\author{Lizhao Zhang}
\address{}
\curraddr{}
\email{lizhaozhang@alum.mit.edu}

\subjclass[2020]{52C25, 52B11, 51M25}

\keywords{rigid motion, continuous deformation, dual, pseudo-Euclidean space.}

\date{}

\begin{abstract}
For any $n$-dimensional simplex in the Euclidean space $\mathbb{R}^n$ with $n\ge 4$,
it is asked that if a continuous deformation preserves the volumes of all the codimension 2 faces,
then is it necessarily a \emph{rigid} motion.
While the question remains open and the general belief is that the answer is affirmative,
for all $n\ge 4$, we provide counterexamples to a variant of the question 
where $\mathbb{R}^n$ is replaced by a pseudo-Euclidean space $\mathbb{R}^{p,n-p}$ 
for some unspecified $p\ge 2$.
\end{abstract}

\maketitle

\section{Introduction}

\subsection{Background and motivations}

For any $n$-dimensional simplex in $\mathbb{R}^n$,
there are $\binom{n+1}{2}$ edges,
and up to congruence the $n$-simplex is uniquely determined by its edge lengths.
As there are also the same number of $(n-2)$-faces,
it is natural to ask the following question:

\begin{question}
\label{question_simplex_rigid}
For any $n$-simplex $Q$ in $\mathbb{R}^n$ with $n\ge 4$, if a continuous deformation preserves 
the $(n-2)$-volumes of all the $(n-2)$-faces of $Q$, then is it necessarily a \emph{rigid} motion?
\end{question}

The discrete version of the question, without the restriction of \emph{continuous} deformation,
asks that if the $(n-2)$-volumes of all the $(n-2)$-faces always determine the $n$-simplex up to \emph{congruence}.
It asks for $n\ge 4$ because the answer for $n=2$ is trivially negative and for $n=3$ is trivially affirmative. 
This question (the combination of both the continuous and the discrete versions)
is classical that comes natural and many people may have asked it in the past,
but it is hard to say who asked it first;
for some account of the background of the question, see \cite{MoharRivin}.
The discrete version was answered negatively for all $n\ge 4$,
and various counterexamples were constructed, see McMullen~\cite{McMullen:simplices}
and Mohar and Rivin~\cite{MoharRivin}.
In fact, the counterexample in \cite{MoharRivin} was not only simple,
but can also be easily extended to construct counterexamples to show that
for any $r\ge 2$ the $r$-volumes of all $r$-faces of an $n$-simplex do not necessarily 
determine the $n$-simplex up to congruence;
and it can be extended to the spherical and hyperbolic spaces as well, 
see Zhang~\cite[Section~2.12]{Zhang:rigidity}.

However, Question~\ref{question_simplex_rigid} remains open,
and a counterexample to Question~\ref{question_simplex_rigid} is, if it were to exist,
a lot harder to find than in the discrete version
(because a continuous counterexample automatically generates a family of infinitely many discrete counterexamples).
The general belief is that the answer to Question~\ref{question_simplex_rigid} may still be affirmative,
e.g., see \cite[Question 3]{MoharRivin} and Sabitov~\cite[Unsolved problem 10]{Sabitov:algebraic}.
For $n=4$, Gaifullin~\cite{Gaifullin:volume} showed that, given a \emph{generic} set of 2-face areas,
there exist not more than a finite congruence classes of 4-simplices in $\mathbb{R}^4$
with the given 2-face areas, and thus no continuous counterexamples in the generic situation.
The proof used an algebraic approach to show that the square of any edge length 
is a root of a polynomial whose coefficients (including the leading term) are polynomials
in these variables of the squares of the 2-face areas, 
and thus only has finite possibilities in the generic situation;
and the method should apply to both Euclidean and pseudo-Euclidean spaces.
But Gaifullin also pointed out that it is much more interesting to obtain results
that hold for \emph{all} simplices, as possible \emph{non}-generic situation 
may arise when the polynomial coefficients (including the leading term) are all zero.
We stress that non-generic situation does not necessarily mean that the simplex has to be degenerate,
so for $n=4$, Question~\ref{question_simplex_rigid} is still not resolved yet.

\subsection{Main results}

While we do not solve Question~\ref{question_simplex_rigid} in this paper, 
but in some aspects contrary to the general belief,
for all $n\ge 4$, we provide counterexamples to a close variant of this question 
where $\mathbb{R}^n$ is replaced by a pseudo-Euclidean space $\mathbb{R}^{p,n-p}$ 
for some unspecified $p\ge 2$.
The treatment is different for $n=4$ and $n\ge 5$.

\begin{theorem}[Main Theorem 1]
\label{theorem_flexible_pseudo_Euclidean_n_4}
For $n=4$ in $\mathbb{R}^{3,1}$,
there exists a continuous family of non-congruent $4$-simplices $Q$ with fixed areas of all the 2-faces,
and all the 2-faces are in Euclidean planes.
\end{theorem}

\begin{theorem}[Main Theorem 2]
\label{theorem_flexible_pseudo_Euclidean}
For any $n\ge 5$, in $\mathbb{R}^{p,n-p}$ for some unspecified $p\ge 2$,
there exists a continuous family of non-congruent $n$-simplices $Q$
with fixed $(n-2)$-volumes of all the $(n-2)$-faces,
and all the dihedral angles are Euclidean angles.
\end{theorem}

Here a \emph{Euclidean angle} means that for an $(n-2)$-face of the $n$-simplex, 
its orthogonal plane in $\mathbb{R}^{p,n-p}$ is a 2-dimensional Euclidean plane, and thus requires $p\ge 2$;
besides, it also ensures that the $(n-2)$-face has non-degenerate metric and thus has non-zero volume.
We stress that being Euclidean angles is neither a necessary condition 
nor a feature that we are particularly looking for, but it helps to present our results.
By contrast, in Theorem~\ref{theorem_flexible_pseudo_Euclidean_n_4} 
the dihedral angles at Euclidean 2-faces in $\mathbb{R}^{3,1}$ are not Euclidean angles.

Though not yet resolving Question~\ref{question_simplex_rigid},
Theorem~\ref{theorem_flexible_pseudo_Euclidean_n_4} and \ref{theorem_flexible_pseudo_Euclidean}
are still somewhat surprising results.
As far as the degrees of freedom is concerned,
the non-positive definiteness of the pseudo-Euclidean space 
does not necessarily make the question easier
because it does not increase the degrees of freedom that we can work on compared to the Euclidean case.
The results may also shed some light, heuristically,
on how to find counterexamples to Question~\ref{question_simplex_rigid} if they were to exist.
But we remark that both Theorem~\ref{theorem_flexible_pseudo_Euclidean_n_4}
and \ref{theorem_flexible_pseudo_Euclidean} are of interest in their own right,
no matter if Question~\ref{question_simplex_rigid} is finally answered positively or negatively in the future.

Rigidity and flexibility of geometric frameworks have been extensively investigated in the past.
A flexible polyhedron in $\mathbb{R}^3$ is a closed polyhedral surface that admits continuous
non-rigid deformation while all faces remaining rigid.
The first embedded (non-self-intersecting) flexible polyhedron in $\mathbb{R}^3$ was discovered by 
Connelly~\cite{Connelly:counterexample}.
It was also shown that the volume of any flexible polyhedron in $\mathbb{R}^3$
remains constant during the continuous deformation, proving the bellows conjecture
(see Sabitov~\cite{Sabitov:invariance} and Connelly \emph{et al.}~\cite{ConnellySabitovWalz}).
Alexandrov~\cite{Alexandrov:Minkowski} showed that flexible polyhedron also exists in Minkowski 3-space
and preserves the volume as well.%
\footnote{The approach of Alexandrov~\cite{Alexandrov:Minkowski} is to apply the existing strategy 
from the Euclidean space and show that it works in the Minkowski space as well.
Though our Theorem~\ref{theorem_flexible_pseudo_Euclidean_n_4} also deals with the Minkowski space,
the approach is different: we start with a strategy not obviously depending on a particular space at first sight,
and then show that it produces counterexamples that can be realized in the Minkowski 4-space,
but whether the strategy can also be adapted to the Euclidean space remains to be seen.
}
The bellows conjecture in the spherical case $\mathbb{S}^3$ was disproved 
by Alexandrov~\cite{Alexandrov:flexible}.
If the facets of an $n$-simplex need not be rigid as above, it was shown that if an $n$-simplex 
(in the Euclidean, spherical or hyperbolic spaces) starts as a 
\emph{degenerate} simplex during a continuous deformation 
with a \emph{single} constraint to preserve a signed sum of its facet volumes
(with the signs naturally induced by Radon's theorem),
then in generic situation its degeneracy is always preserved,
with the non-generic situation also precisely specified (Zhang~\cite{Zhang:lifting}).

For any $n$-simplex in $\mathbb{R}^{p,n-p}$ with fixed volumes of codimension 2 faces
during a continuous deformation,
in the spirit of the bellows conjecture, it seems natural to ask that if the volume of the $n$-simplex 
itself also necessarily remains constant.
(A discrete version of the question in $\mathbb{R}^n$ due to Connelly was already answered 
negatively for all $n\ge 4$.)
If $n$ is even and $n\ge 6$, while we do not fully answer this question, 
for what it is worth, we observe that, 
for the \emph{particular} counterexamples we constructed
in Theorem~\ref{theorem_flexible_pseudo_Euclidean}, 
the volume of the $n$-simplex \emph{does} remain constant.
We answer this question negatively if $n=4$ or $n$ is odd.

\begin{theorem}
\label{theorem_flexible_volume_non_constant}
If $n=4$ or $n\ge 5$ and $n$ is odd, in $\mathbb{R}^{p,n-p}$ for some unspecified $p\ge 2$,
there exists a continuous family of non-congruent $n$-simplices $Q$
with fixed $(n-2)$-volumes of all the $(n-2)$-faces,
but the volume of $Q$ does not remain constant.
\end{theorem}
 
Similar to Theorem~\ref{theorem_flexible_pseudo_Euclidean_n_4},
in $\mathbb{R}^{p,5-p}$ for some unspecified $p\ge 2$,
we can also construct a continuous family of non-congruent $5$-simplices $Q$ 
with fixed areas of all the 2-faces.

\begin{theorem}
\label{theorem_flexible_simplex_5_face_2}
For $n=5$, in $\mathbb{R}^{p,5-p}$ for some unspecified $p\ge 2$,
there exists a continuous family of non-congruent $5$-simplices $Q$ with fixed areas of all the 2-faces,
and all the 2-faces are in Euclidean planes, but the volume of $Q$ does not remain constant.
\end{theorem}

While the main theme in this paper is about codimension \emph{two} faces,
Theorem~\ref{theorem_flexible_simplex_5_face_2} is somewhat an even more surprising result than 
Theorem~\ref{theorem_flexible_pseudo_Euclidean_n_4},
because the number of volume constraints, which is $\binom{6}{3}=20$, 
is even \emph{bigger} than the degrees of freedom that is $\binom{6}{2}=15$. 
But the method we use in Theorem~\ref{theorem_flexible_simplex_5_face_2} 
is almost the same as that of Theorem~\ref{theorem_flexible_pseudo_Euclidean_n_4}
without much extra effort.
Gaifullin~\cite[Corollary 1.3]{Gaifullin:volume} showed that for a \emph{generic} simplex in
$\mathbb{R}^n$ with $n\ge 5$, any simplex with the same areas of the corresponding 2-faces
is congruent to it, and the algebraic approach should work for the pseudo-Euclidean spaces as well,
except that there may be a finite congruence classes instead of just one congruence class.
Then again, our counterexamples for $n=5$ in $\mathbb{R}^{p,5-p}$ 
must happen only at the \emph{non}-generic situation, in the sense of Gaifullin.

\subsection{Why \emph{pseudo}-Euclidean spaces}

For any $n$-dimensional simplex $Q$ in $\mathbb{R}^n$, by the Cayley--Menger determinant, 
the square of the volume of any face is a polynomial in these variables of the squares of the edge lengths,
where the variables are algebraically independent.
In an algebraic sense, if those variables are allowed to be any real values (including negative numbers and zero), 
though $Q$ may not be realized in $\mathbb{R}^n$,
we can still use the same polynomial (but then need to take the absolute value)
to \emph{define} the ``square of the volume'' of the face of $Q$.
(Once $Q$ can be realized in some, in fact, \emph{any}, pseudo-Euclidean space,
then this definition of the ``volume'' of the face of $Q$ will be justified.)
This inspires us, for both Theorem~\ref{theorem_flexible_pseudo_Euclidean_n_4}
and \ref{theorem_flexible_pseudo_Euclidean},
to first focus on assigning some appropriate values (maybe negative or zero) to those variables
such that the ``volumes'' of all codimension 2 faces remain constant during the deformation,
and then concern about in what pseudo-Euclidean space the  $n$-simplices can be realized.

The main difference between the Euclidean and the pseudo-Euclidean space is that,
as far as this paper is concerned, instead of working on positive semi-definite matrices,
we extend the domain to symmetric matrices.
Beside a definition of the metric and the volume in the pseudo-Euclidean space,
no knowledge of pseudo-Riemannian geometry is assumed of the reader.

\section{Preliminaries}

\subsection{Some notions}

For $0\le p\le n$, $\mathbb{R}^{p,n-p}$ is an $n$-dimensional vector space endowed with a bilinear product
\[x\cdot y=x_1y_1+\cdots +x_py_p-x_{p+1}y_{p+1}-\cdots -x_ny_n,
\]
which induces a metric $ds^2=dx_1^2+\cdots +dx_p^2-dx_{p+1}^2-\cdots -dx_n^2$.

For a $k$-simplex $G$ in $\mathbb{R}^{p,n-p}$, choose any vertex as a base point, 
let $u_i$ $(1\le i\le k)$ be the vectors from this vertex to other $k$ vertices.

\begin{definition}
\label{definition_square_length}
We refer to $u_i^2$ (as well as $(u_i-u_j)^2$) as the \emph{square of the edge length},
for both the Euclidean and pseudo-Euclidean cases,
no matter $u_i^2$ is positive, negative or zero.
\end{definition}

The \emph{length} of the edge, $(u_i^2)^{1/2}$, is either a non-negative number or the product of
a non-negative number by the imaginary unit $i$.
But in this paper we almost never use the notion of the \emph{length} alone 
and always use the \emph{square} of the length instead,
so we do not need to worry about the imaginary unit $i$.
For $G$ (with the base point chosen), we obtain the corresponding Gram matrix
\begin{equation}
\label{equation_simplex_matrix}
(u_i\cdot u_j)_{1\le i,j\le k}.
\end{equation}
As
\begin{equation}
\label{equaiton_dot_product_edge_squared}
(u_i-u_j)^2=u_i^2+u_j^2-2u_i\cdot u_j,
\end{equation}
so by the matrix $(u_i\cdot u_j)_{1\le i,j\le k}$ itself, without knowing more details of $u_i$,
we can recover all the information of the squares of the edge lengths of $G$.
As a result,  the Gram determinant $\det(u_i\cdot u_j)_{1\le i,j\le k}$ 
is a polynomial in the variables of the squares of the edge lengths of $G$.
Due to the non-positive definiteness of $\mathbb{R}^{p,n-p}$, unlike in the Euclidean case,
here $\det(u_i\cdot u_j)_{1\le i,j\le k}$ may be negative or zero.

\begin{definition}
\label{definition_volume}
The \emph{volume}, or the $k$-volume, of $G$ is defined by
\begin{equation}
\label{equation_volume_pseudo_Euclidean}
V_k(G):=|\det(u_i\cdot u_j)_{1\le i,j\le k}|^{1/2}/k!.
\end{equation}
\end{definition}

This definition of the volume is the same as the integral of the ``absolute value'' 
of the induced volume element on the domain, 
e.g., if $k=n$, then the $n$-volume of $G$ is the integral of $dx_1\cdots dx_n$ over the domain of $G$.
If the induced metric on the domain is degenerate, then the volume is zero. 

\begin{definition}
Let the \emph{signed square of the volume} of $G$ be $\det(u_i\cdot u_j)_{1\le i,j\le k}/(k!)^2$.
\end{definition}

We add \emph{signed} to the name to address the sign issue of $\det(u_i\cdot u_j)_{1\le i,j\le k}$.
By convention, for $k=2$ the notion of 2-volume is also interchangeable with the notion of ``area''.

\begin{remark}
We caution that for $k=1$, the notion of 1-volume is \emph{not} the same as ``length'' 
in this paper (they are the same in the Euclidean case), but the absolute value of the length instead.
This is because as $u_i^2$'s are variables to calculate the higher dimensional volumes,
it is more natural to refer to $u_i^2$, by convention, as the ``square of the edge length'' 
(as we did in Definition~\ref{definition_square_length}),
rather than rename it to the ``\emph{signed} square of the edge length''.
To summarize, for $k=1$, the ``signed square of the 1-volume'' is the same as the square of the length.
\end{remark}

The volume is additive and invariant under isometries of $\mathbb{R}^{p,n-p}$,
and the volume of $G$ (\ref{equation_volume_pseudo_Euclidean})
is independent of the base point chosen.
Our definitions of the length and the volume agree with the definitions in 
Alexandrov~\cite{Alexandrov:Minkowski},
which were introduced to analyze flexible polyhedron in the Minkowski space;
and in that paper a notion of the \emph{angle} in the Minkowski 2-plane was also introduced,
but for our purpose we do not need this notion of \emph{Minkowski angle}.

\subsection{Basic properties}

In the Euclidean space $\mathbb{R}^n$, if a $k$-simplex is non-degenerate,%
\footnote{By degenerate we mean that the vertices of the $k$-simplex 
are confined to a lower $(k-1)$-dimensional plane.
}
then the matrix $(u_i\cdot u_j)_{1\le i,j\le k}$ is positive definite and 
$\det(u_i\cdot u_j)_{1\le i,j\le k}$ is positive.

\begin{remark}
\label{remark_zero_volume_non_degenerate}
In $\mathbb{R}^{p,n-p}$, if $k<n$, we caution that if $\det(u_i\cdot u_j)_{1\le i,j\le k}$ is zero,
though it implies that $G$ has zero volume, $G$ may still be \emph{non}-degenerate.
This happens when $G$ is contained in a $k$-dimensional plane (in $\mathbb{R}^{p,n-p}$)
that has degenerate metric.
But if $\det(u_i\cdot u_j)_{1\le i,j\le k}$ is non-zero, 
then $G$ is always non-degenerate.
\end{remark}

The following property in the reverse direction will be used throughout this paper. 

\begin{lemma}
\label{lemma_symmetric_matrix_decomposition}
For every real symmetric $n\times n$ matrix $C$, if $\det(C)$ is non-zero,
then $C$ may be treated as the Gram matrix $(u_i\cdot u_j)_{1\le i,j\le n}$
of some linearly independent vectors $u_1, \dots, u_n$ of a pseudo-Euclidean space $\mathbb{R}^{p,n-p}$,
where $p$ is the number of positive eigenvalues of $C$.
Particularly if $C$ is positive definite, then $p=n$.
\end{lemma}

\begin{proof}
As $C$ is symmetric, so $C$ may be decomposed as $B^TDB$ where $B$ is an orthogonal matrix $B^TB=I$,
and $D$ is a diagonal matrix of the (real) eigenvalues of $C$.
As $\det(C)$ is non-zero, so all the eigenvalues of $C$ are non-zero
and we assume the diagonal entries of $D$ contain $p$ positive 
followed by $n-p$ negative entries for some $p$. 
If $C$ is positive definite, then it is well known that $p=n$.
By scaling, $D$ can be written as $S^TD_0S$,
where $D_0$ is a diagonal matrix whose diagonal entries are 1's followed by $(-1)$'s,
and $S$ is also a diagonal matrix. Then $C$=$(SB)^TD_0(SB)$.
Let the $n$ columns of $SB$ be $u_1,\dots,u_n$,
then $C=(u_i\cdot u_j)_{1\le i,j\le n}$, where $u_i\cdot u_j:=u_i^TD_0u_j$,
and $u_1,\dots,u_n$ are linearly independent vectors in $\mathbb{R}^{p,n-p}$.
\end{proof}

\begin{remark}
\label{remark_symmetric_matrix_decomposition}
If $\det(C)=0$, we can modify the proof above slightly and show that $C$ 
can still be realized as a Gram matrix of vectors $u_1, \dots, u_n$ of some pseudo-Euclidean space,
but the $p$ may not be uniquely determined.
\end{remark}

We have the following simple but important observation, which is crucial for this paper.

\begin{lemma}
\label{lemma_matrix_2_by_2}
For $2\times 2$ matrices $A=\begin{pmatrix} a_1^2 & a_1a_2 \\ a_1a_2 & a_2^2 \end{pmatrix}$
and $B=\begin{pmatrix} b_1^2 & b_1b_2 \\ b_1b_2 & b_2^2 \end{pmatrix}$, 
let $C=tA+\frac{1}{t}B$. Then $\det(C)$ is independent of the value of $t>0$, 
and is positive unless $a_1b_2=a_2b_1$.
\end{lemma}

\begin{proof}
Because $\det(A)=0$ and $\det(B)=0$, so $\det(C)$ does not have the $t^2$ and $\frac{1}{t^2}$ terms.
This leaves $\det(C)$ with only the constant term (as a product of $t$ and $\frac{1}{t}$)
\[a_1^2b_2^2+a_2^2b_1^2-2a_1a_2b_1b_2=(a_1b_2-a_2b_1)^2,
\]
which is positive unless $a_1b_2=a_2b_1$.
\end{proof}

\begin{remark}
\label{remark_geometric_interpretation}
Here is a geometric interpretation of Lemma~\ref{lemma_matrix_2_by_2}.
By the discussion above,  unless $a_1b_2=a_2b_1$, for $t>0$ the $2\times 2$ matrix $C$ is positive definite.
For any fixed $t>0$, 
$C$ determines a Euclidean triangle up to congruence.
Then Lemma~\ref{lemma_matrix_2_by_2} means that, though the shape of the triangle varies over $t$,
the square of the area of the triangle is independent of the value of $t$.
\end{remark}

Lemma~\ref{lemma_matrix_2_by_2} plays a key role in the proofs of
both Theorem~\ref{theorem_flexible_pseudo_Euclidean_n_4}
and \ref{theorem_flexible_pseudo_Euclidean},
where both proofs are essentially built on how to apply Lemma~\ref{lemma_matrix_2_by_2} properly
to find counterexamples that can be decomposed into two degenerate components.
The meanings of both \emph{decompose} and \emph{degenerate} are problem-specific.

The paper is organized as follows.
We apply Lemma~\ref{lemma_matrix_2_by_2}
to prove Theorem~\ref{theorem_flexible_pseudo_Euclidean_n_4}
(as well as Theorem~\ref{theorem_flexible_simplex_5_face_2}) in a more direct way.
For Theorem~\ref{theorem_flexible_pseudo_Euclidean},
we first transform it into equivalent questions about the \emph{dual} of the $n$-simplex,
and then apply Lemma~\ref{lemma_matrix_2_by_2} \emph{indirectly} to prove 
Theorem~\ref{theorem_flexible_pseudo_Euclidean}.
The proofs of Theorem~\ref{theorem_flexible_pseudo_Euclidean_n_4}
and \ref{theorem_flexible_pseudo_Euclidean} are constructive for all $n\ge 4$.

\section{Proofs of Theorem~\ref{theorem_flexible_pseudo_Euclidean_n_4}
and \ref{theorem_flexible_simplex_5_face_2}}

Recall that an $n$-simplex in $\mathbb{R}^{p,n-p}$, with a base point chosen, 
corresponds to an $n\times n$ Gram matrix $(u_i\cdot u_j)_{1\le i,j\le n}$ 
(symmetric, but not necessarily positive-definite) as in (\ref{equation_simplex_matrix}).
By (\ref{equaiton_dot_product_edge_squared}) the matrix also determines the 
squares of all the edge lengths of the simplex (and vice versa), and the relationship is \emph{linear}.
This linear relationship can be extended to be between \emph{any} symmetric matrix
and what we call a \emph{pseudo $n$-simplex} (with a base point chosen and the vertices ordered).
We can assign any real values (including negative numbers and zero) 
to the squares of the edge lengths of a pseudo simplex,
without concerning about in what pseudo-Euclidean space it can be embedded.
A simplex in $\mathbb{R}^{p,n-p}$ is automatically a pseudo simplex.
Just like symmetric matrices, pseudo simplices form a vector space.
Namely, with respect to the matrix operations, pseudo simplices can take a \emph{summation} by adding 
the squares of the corresponding edge lengths together,
or similarly, take a product with a scalar $t$.
Compared to the corresponding symmetric matrix, 
pseudo simplex is particularly helpful to ``visualize'' the faces that do not contain the base point.

By Lemma~\ref{lemma_symmetric_matrix_decomposition}
and Remark~\ref{remark_symmetric_matrix_decomposition},
every pseudo simplex can be realized in some pseudo-Euclidean space.
As the volume of the realized simplex is completely determined 
by the squares of the edge lengths and does not depend on the particular embedding,
we can use the same formula (but then need to take the absolute value)
to assign a ``volume'' to the pseudo simplex, and for its faces as well
(see Definition~\ref{definition_volume}).

\subsection{Proof of Theorem~\ref{theorem_flexible_pseudo_Euclidean_n_4}}

Our approach to prove Theorem~\ref{theorem_flexible_pseudo_Euclidean_n_4}
is to first find two pseudo 4-simplices $Q_1$ and $Q_2$ with the properties:
(1) all the 2-faces of $Q_1$ and $Q_2$ have zero areas, and
(2) the summation of $tQ_1$ and $\frac{1}{t}Q_2$ for $t>0$, denoted by $Q$, 
has non-zero volume and can be realized in $\mathbb{R}^{3,1}$ (except for some $t>0$).%
\footnote{If $Q$ can be realized in $\mathbb{R}^4$ instead (replacing $\mathbb{R}^{3,1}$),
then we would have found counterexamples to Question~\ref{question_simplex_rigid} for $n=4$.
}

\begin{remark}
If both $Q_1$ and $Q_2$ can be realized as ``five points on a line'',
then they easily satisfy condition (1), but not (2) as they both correspond to $4\times 4$ matrices with rank 1,
then $Q$ corresponds to a $4\times 4$ matrix whose rank is at most 2 and thus $Q$ has zero volume.
So we need to search for more ``complicated'' $Q_1$ and $Q_2$.
\end{remark}

We first construct a pseudo 3-simplex $Q_0$ whose 2-faces all have zero-areas but its volume is non-zero,
which will be used to construct pseudo 4-simplices $Q_1$ and $Q_2$ later.

\begin{example}
\label{example_simplex_3}
For some non-zero $a_1$, $a_2$ and $a_3$ satisfying $a_1+a_2+a_3=0$, 
let $Q_0$ be a pseudo 3-simplex (with the vertices numbered from 0 to 3)
who has equal squares of edge lengths on opposite sides, 
with values of $a_1^2$, $a_2^2$ and $a_3^2$ 
(they are also the 3 sides of each 2-face, see Figure~\ref{figure_simplex_3}).
It can be checked that all the 2-faces have zero areas, but the volume of $Q_0$ is \emph{non}-zero,
which is crucial for the construction.
\end{example}

\begin{figure}[h]
\centering
\resizebox{.25\textwidth}{!}
 {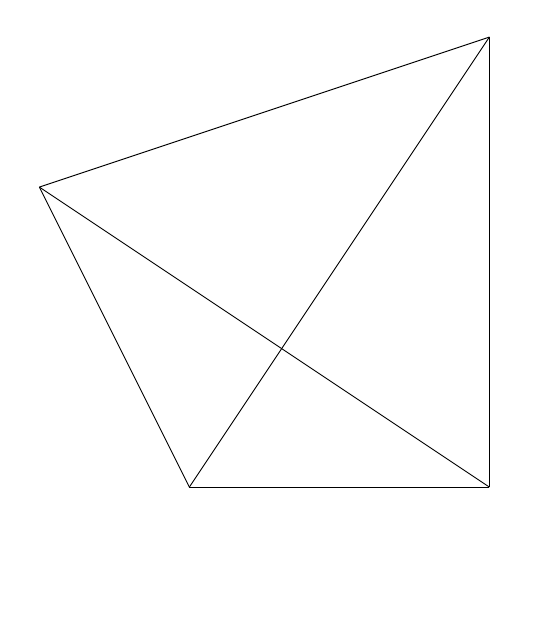}
\caption{Pseudo 3-simplex $Q_0$ with squares of edge lengths}
\label{figure_simplex_3}
\end{figure}

Guided by Example~\ref{example_simplex_3},
now for some non-zero $a_1$, $a_2$ and $a_3$ satisfying $a_1+a_2+a_3=0$,
let $Q_1$ be a pseudo 4-simplex (with the vertices numbered from 0 to 4, 
but 0 and 4 are placed on the same node)
whose squares of edge lengths are labeled in Figure~\ref{figure_simplex_4} (a).
Similarly for some non-zero $b_2$, $b_3$ and $b_4$ satisfying $b_2+b_3+b_4=0$, 
let $Q_2$ be a pseudo 4-simplex (with the vertices also numbered from 0 to 4, 
but 0 and 1 are placed on the same node) 
whose squares of edge lengths are labeled in Figure~\ref{figure_simplex_4} (b).

\begin{figure}[h]
\centering
\resizebox{.6\textwidth}{!}
  {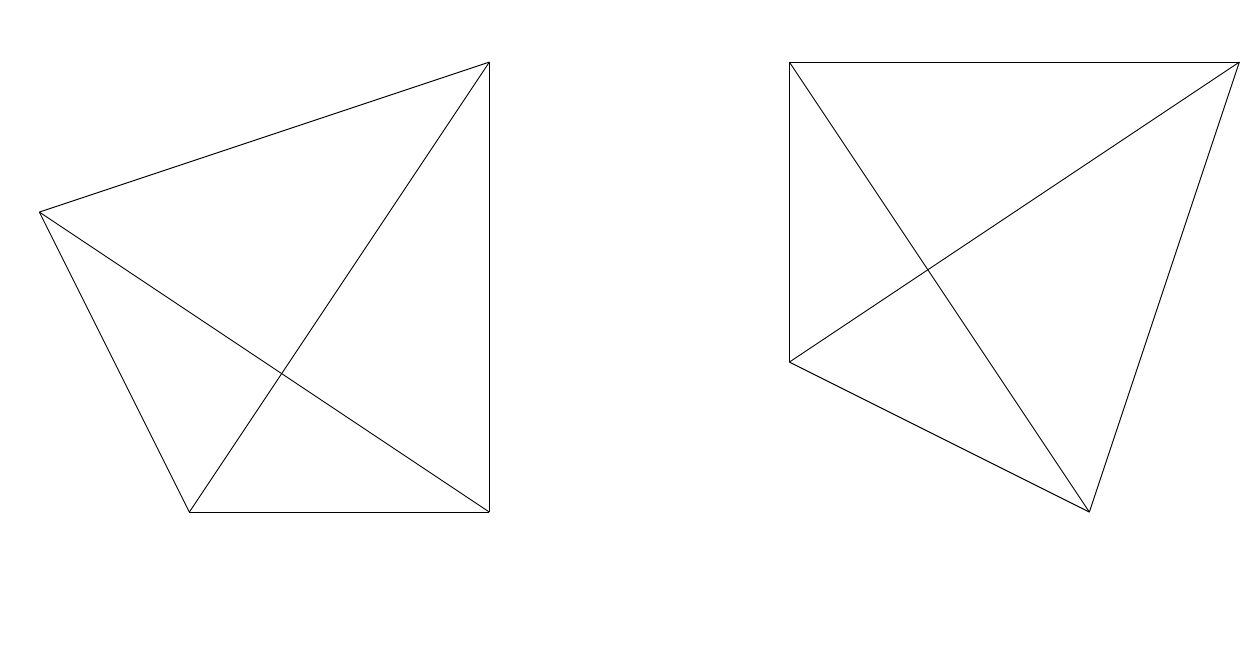}
\caption{Pseudo 4-simplices $Q_1$ and $Q_2$ with squares of edge lengths}
\label{figure_simplex_4}
\end{figure}

Choose vertex 0 as the base point for both $Q_1$ and $Q_2$.
As $a_1+a_2+a_3=0$, so
\[(a_1^2+a_2^2-a_3^2)/2=(a_1^2+a_2^2-(-a_1-a_2)^2)/2=-a_1a_2, 
\]
and so on (and similarly for $b_2$, $b_3$ and $b_4$).
Thus the corresponding matrices of $Q_1$ and $Q_2$ are 
\begin{equation*}
{\footnotesize
A=
\begin{pmatrix} 
a_1^2 & -a_1a_2 & -a_1a_3 & 0 \\ 
-a_1a_2 & a_2^2 & -a_2a_3 & 0 \\
-a_1a_3 & -a_2a_3 & a_3^2 & 0 \\
0 & 0 & 0 & 0
\end{pmatrix}
\quad\text{and}\quad
B=
\begin{pmatrix} 
0 & 0 & 0 & 0 \\
0 & b_2^2 & -b_2b_3 & -b_2b_4 \\
0 & -b_2b_3 & b_3^2 & -b_3b_4 \\
0 & -b_2b_4 & -b_3b_4 & b_4^2 
\end{pmatrix}
.
}
\end{equation*}

It is easy to check that all the 2-faces of $Q_1$ and $Q_2$ have zero areas.
Now let a pseudo 4-simplex $Q$ be $tQ_1+\frac{1}{t}Q_2$ for $t>0$,
then its corresponding matrix is $tA+\frac{1}{t}B$, which we denote by $C$.

\begin{lemma}
\label{lemma_flexible_pseudo_4}
For some appropriate non-zero $a_1$, $a_2$ and $a_3$ satisfying $a_1+a_2+a_3=0$
and non-zero $b_2$, $b_3$ and $b_4$ satisfying $b_2+b_3+b_4=0$, 
except for some $t>0$, all the 2-faces of $Q$ have fixed positive signed squares of the areas,
and $\det(C)$ is negative and not a constant over $t$.
\end{lemma}

\begin{proof}
As all the 2-faces of $Q_1$ and $Q_2$ have zero areas,
so by Lemma~\ref{lemma_matrix_2_by_2} and Remark~\ref{remark_geometric_interpretation},
all the 2-faces of $Q$ have fixed \emph{non-negative} signed squares of the areas for $t>0$.
To make them all positive, e.g., for the 2-face with the vertices 0, 2, and 3, 
(by Lemma~\ref{lemma_matrix_2_by_2} again) we need to exclude the case of $a_2b_3=a_3b_2$, and so on.
Once those cases are excluded, then all the 2-faces of $Q$
have fixed \emph{positive} signed squares of the areas.

As $C=tA+\frac{1}{t}B$, so $\det(C)$ can be expressed as a sum of signed products
such that each summand is the product of a minor of $tA$ and the 
``complement'' minor of $\frac{1}{t}B$.
As $\det(A)=\det(B)=0$, so in $\det(C)$ there are no $t^4$ and $\frac{1}{t^4}$ terms.
The $t^2$ term is the product of the upper left principal minor (of order 3) of $tA$ and $\frac{b_4^2}{t}$,
which is $-4a_1^2a_2^2a_3^2b_4^2t^2$.
Similarly the $\frac{1}{t^2}$ term is the produce of $a_1^2t$ and
the lower right principal minor (of order 3) of $\frac{1}{t}B$,
which is $-4a_1^2b_2^2b_3^2b_4^2\frac{1}{t^2}$.
The constant term, with only two equal non-zero summands, is 
\[-2\det\begin{pmatrix} a_1^2 & -a_1a_2 \\ -a_1a_3 & -a_2a_3 \end{pmatrix}
\cdot\det\begin{pmatrix} -b_2b_3 & -b_2b_4 \\ -b_3b_4 & b_4^2 \end{pmatrix}
=-8a_1^2a_2a_3b_2b_3b_4^2.
\]
So
\begin{align*}
\det(C)
&=-4a_1^2a_2^2a_3^2b_4^2t^2-4a_1^2b_2^2b_3^2b_4^2\frac{1}{t^2}
-8a_1^2a_2a_3b_2b_3b_4^2 \\
&=-4a_1^2b_4^2(a_2a_3t+b_2b_3\frac{1}{t})^2,
\end{align*}
and except for $t^2=-\frac{b_2b_3}{a_2a_3}$,
$\det(C)$ is negative and is not a constant over $t$.
\end{proof}

We are now ready to prove Theorem~\ref{theorem_flexible_pseudo_Euclidean_n_4}, which we recall below.

\begingroup
\def\thetheorem{\ref{theorem_flexible_pseudo_Euclidean_n_4}}
\begin{theorem}
For $n=4$ in $\mathbb{R}^{3,1}$,
there exists a continuous family of non-congruent $4$-simplices $Q$ with fixed areas of all the 2-faces,
and all the 2-faces are in Euclidean planes.
\end{theorem}
\addtocounter{theorem}{-1}
\endgroup

\begin{proof}
For the family of pseudo $4$-simplices $Q$ we construct above, 
by Lemma~\ref{lemma_flexible_pseudo_4}, $\det(C)$ is negative except for some $t>0$.
So by Lemma~\ref{lemma_symmetric_matrix_decomposition},
for a continuous interval of $t$,
$Q$ can be realized as a family of $4$-simplices (we still call it $Q$) 
in a pseudo-Euclidean space $\mathbb{R}^{p,4-p}$ 
where $p$ is the number of positive eigenvalues of $C$.
By Lemma~\ref{lemma_flexible_pseudo_4} again, 
all the 2-faces of $Q$ have fixed positive signed squares of the areas,
so they are in Euclidean planes and thus $p\ge 2$.
Since $\det(C)$ is negative and is the product of the (real) eigenvalues of $C$, 
so $4-p$, the number of negative eigenvalues of $C$, is odd. As $p\ge 2$, so $p$ is 3. 
Therefore $Q$ can be realized in $\mathbb{R}^{3,1}$. 
\end{proof}

Notice also that $\det(C)$ is not a constant over $t$, so the volume of $Q$ does not remain constant,
proving Theorem~\ref{theorem_flexible_volume_non_constant} for $n=4$.

\begin{corollary}[Theorem~\ref{theorem_flexible_volume_non_constant}, for $n=4$]
\label{corollary_flexible_volume_non_constant_4}
For $n=4$ in $\mathbb{R}^{3,1}$,
there exists a continuous family of non-congruent $4$-simplices $Q$
with fixed areas of all the 2-faces,
but the volume of $Q$ does not remain constant.
\end{corollary}

\subsection{Proof of Theorem~\ref{theorem_flexible_simplex_5_face_2}}

The construction below is almost the same as above without much change,
but we need to redefine some of the notations.

For some non-zero $a_1$, $a_2$ and $a_3$ satisfying $a_1+a_2+a_3=0$,
let $Q_1$ be a pseudo 5-simplex (with the vertices numbered from 0 to 5, 
but 0 and 5 are placed on the same node and 3 and 4 are placed on the same node)
whose squares of edge lengths are labeled in Figure~\ref{figure_simplex_5} (a).
Similarly for some non-zero $b_2$, $b_4$ and $b_5$ satisfying $b_2+b_4+b_5=0$, 
let $Q_2$ be a pseudo 5-simplex (with the vertices also numbered from 0 to 5, 
but 0 and 1 are placed on the same node and 2 and 3 are placed on the same node)
whose squares of edge lengths are labeled in Figure~\ref{figure_simplex_5} (b).
 
\begin{figure}[h]
\centering
\resizebox{.6\textwidth}{!}
  {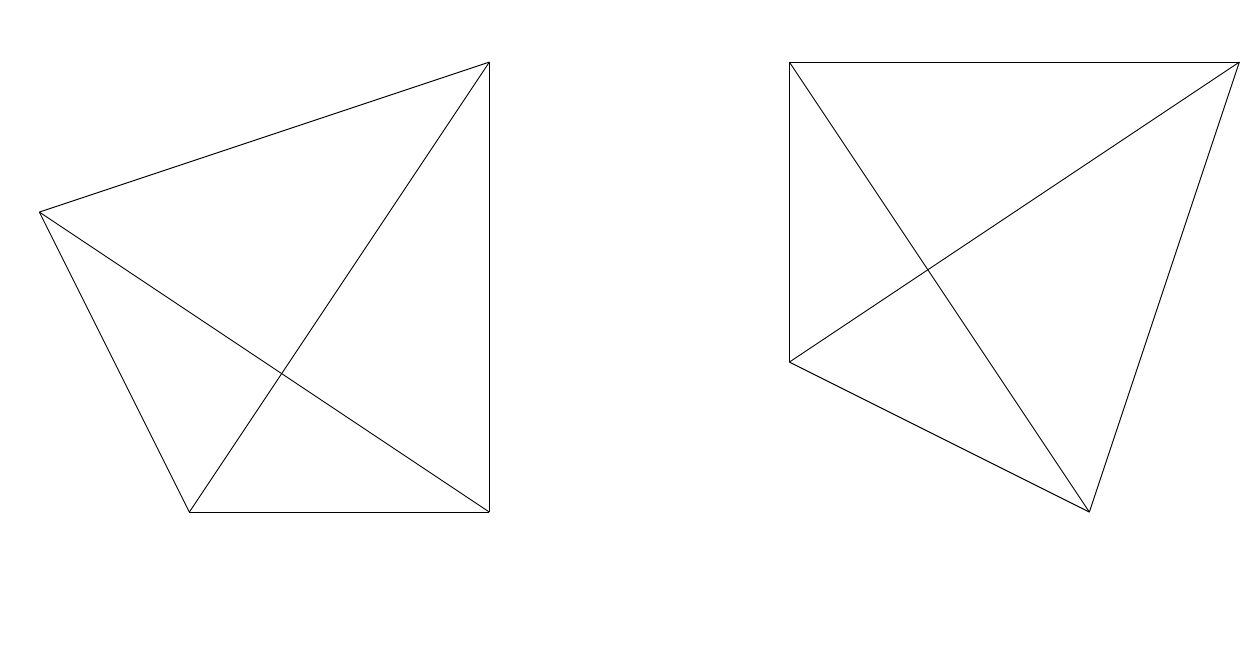}
\caption{Pseudo 5-simplices $Q_1$ and $Q_2$ with squares of edge lengths}
\label{figure_simplex_5}
\end{figure}

Choose vertex 0 as the base point for both $Q_1$ and $Q_2$.
As $a_1+a_2+a_3=0$ and $b_2+b_4+b_5=0$,
the corresponding matrices of $Q_1$ and $Q_2$ are 
\begin{equation*}
{\footnotesize
A=
\begin{pmatrix} 
a_1^2 & -a_1a_2 & -a_1a_3 & -a_1a_3 & 0 \\ 
-a_1a_2 & a_2^2 & -a_2a_3 & -a_2a_3 & 0 \\
-a_1a_3 & -a_2a_3 & a_3^2 & a_3^2 & 0 \\
-a_1a_3 & -a_2a_3 & a_3^2 & a_3^2 & 0 \\
0 & 0 & 0 & 0 & 0
\end{pmatrix}
\quad\text{and}\quad
B=
\begin{pmatrix} 
0 & 0 & 0 & 0 & 0 \\
0 & b_2^2 & b_2^2 & -b_2b_4 & -b_2b_5 \\
0 & b_2^2 & b_2^2 & -b_2b_4 & -b_2b_5 \\
0 & -b_2b_4 & -b_2b_4 & b_4^2 & -b_4b_5 \\
0 & -b_2b_5 & -b_2b_5 & -b_4b_5 & b_5^2 
\end{pmatrix}
.
}
\end{equation*}

As before, all the 2-faces of $Q_1$ and $Q_2$ have zero areas.
Now let a pseudo 5-simplex $Q$ be $tQ_1+\frac{1}{t}Q_2$ for $t>0$,
then its corresponding matrix is $tA+\frac{1}{t}B$, which we denote by $C$.

\begin{lemma}
\label{lemma_flexible_pseudo_5}
For some appropriate non-zero $a_1$, $a_2$ and $a_3$ satisfying $a_1+a_2+a_3=0$
and non-zero $b_2$, $b_4$ and $b_5$ satisfying $b_2+b_4+b_5=0$, except for some $t>0$,
all the 2-faces of $Q$ have fixed positive signed squares of the areas,
and $\det(C)$ is non-zero and not a constant over $t$.
\end{lemma}

\begin{proof}
As before, once some cases (of the values of $a_i$'s and $b_j$'s) are excluded,
then all the 2-faces of $Q$ have fixed positive signed squares of the areas.

As $C=tA+\frac{1}{t}B$, so $\det(C)$, 
compared to how we computed in the previous section for $n=4$, can be shown to be
\[\det(C)=-16a_1^2a_2a_3b_2b_4b_5^2(a_2a_3t+b_2b_4\frac{1}{t}).
\]
Except for $t^2=-\frac{b_2b_4}{a_2a_3}$, $\det(C)$ is non-zero and not a constant over $t$.
\end{proof}

For the family of pseudo $5$-simplices $Q$ we construct above, 
by Lemma~\ref{lemma_flexible_pseudo_5}, $\det(C)$ is non-zero except for some $t>0$.
So by Lemma~\ref{lemma_symmetric_matrix_decomposition}, for a continuous interval of $t$,
$Q$ can be realized as a family of $5$-simplices (we still call it $Q$)
in a pseudo-Euclidean space $\mathbb{R}^{p,5-p}$ for some $p$.
By Lemma~\ref{lemma_flexible_pseudo_5} again, 
all the 2-faces of $Q$ have fixed positive signed squares of the areas,
so they are in Euclidean planes and thus $p\ge 2$.
As $\det(C)$ is not a constant over $t$, so the volume of $Q$ does not remain constant.
This proves Theorem~\ref{theorem_flexible_simplex_5_face_2}, which we recall below.

\begingroup
\def\thetheorem{\ref{theorem_flexible_simplex_5_face_2}}
\begin{theorem}
For $n=5$, in $\mathbb{R}^{p,5-p}$ for some unspecified $p\ge 2$,
there exists a continuous family of non-congruent $5$-simplices $Q$ with fixed areas of all the 2-faces,
and all the 2-faces are in Euclidean planes, but the volume of $Q$ does not remain constant.
\end{theorem}
\addtocounter{theorem}{-1}
\endgroup

\section{Equivalent questions of Question~\ref{question_simplex_rigid}}

To prove Theorem~\ref{theorem_flexible_pseudo_Euclidean}
(about an $n$-simplex $Q$ in $\mathbb{R}^{p,n-p}$ for some $p$),
we first reformulate Question~\ref{question_simplex_rigid} 
(about an $n$-simplex $Q$ in $\mathbb{R}^n$)
to an equivalent form about the \emph{dual} of $Q$ in $\mathbb{R}^n$,
and this process also (self-)explains that how we come to formulate 
Theorem~\ref{theorem_flexible_pseudo_Euclidean} in the first place.

For an $n$-simplex $Q$ in $\mathbb{R}^n$, without loss of generality,
we assume that the origin $O$ is the centroid of the vertices of $Q$ 
(or simply call it the centroid of $Q$).
This assumption is not essential for $Q$ 
because the volume of any face of $Q$ is invariant under translation, 
but becomes more important when the dual of $Q$ is introduced next.
For $1\le i\le n+1$, denote $F_i$ an $(n-1)$-face of $Q$,
and $F_{ij}$ ($i\ne j$) the $(n-2)$-face on the intersection of $F_i$ and $F_j$.

For some $c>0$, let an $n$-simplex $P$ be a \emph{dual} of $Q$ 
\begin{equation}
\label{equation_dual_Euclidean}
P:=\{y \in \mathbb{R}^n: x\cdot y \le c \quad\text{for all $x\in Q$}\}.
\end{equation}
When $c=1$, $P$ is the \emph{polar dual} $Q^{\ast}$ of $Q$.
Here $c$ may vary when $Q$ varies, but the value of $c$ is not important and can be adjusted later.
Let the vertices of $P$ be $P_i$ and $v_i=\overrightarrow{OP_i}$,
then $v_i$ is an outward normal vector to $Q$ at $F_i$.
As $O$ is the centroid of $Q$, it can be shown that $\norm{v_i}$, defined by $|v_i^2|^{1/2}$,
is proportional to the $(n-1)$-volume of $F_i$.
The Minkowski relation for the facet volumes of $Q$ states that $\sum v_i=0$, 
so $O$ is also the centroid of $P$.

It is well known that the area of the triangle $OP_iP_j$ is proportional to the $(n-2)$-volume of $F_{ij}$
(e.g., see Lee~\cite[Theorem 19]{Lee:stress})
\[V_{n-2}(F_{ij})=\frac{c_2V_2(OP_iP_j)}{V_n(P)}\cdot c^{n-2},
\]
where $c_2$ is a constant that only depends on $n$.
So Question~\ref{question_simplex_rigid} is equivalent to the following question about $P$.

\begin{question}
\label{question_simplex_triangle_rigid}
For any $n$-simplex $P$ in $\mathbb{R}^n$ ($n\ge 4$) with the centroid fixed at the origin $O$, 
if a continuous deformation preserves the areas of all triangles $OP_iP_j$,
then is it necessarily a rigid motion?
\end{question}

Note that the square of the area of the triangle $OP_iP_j$, multiplied by a factor of 4, 
is the determinant of the $2\times 2$ matrix
\begin{equation}
\begin{pmatrix} v_i^2 & v_i\cdot v_j \\ v_i\cdot v_j & v_j^2 \end{pmatrix},
\end{equation}
where $v_i\cdot v_j$ is the inner product of $v_i$ and $v_j$ in $\mathbb{R}^n$.
Conversely, this $2\times 2$ matrix, without knowing more details of $v_i$ and $v_j$, 
can recover $\overrightarrow{P_iP_j}^2$ as $v_i^2+v_j^2-2v_i\cdot v_j$.

If we treat $v_i$ as a column $n$-vector and let $V=(v_1,\dots, v_{n+1})$ be an $n\times (n+1)$ matrix
and $U:=V^TV=(v_i^T v_j)_{1\le i,j\le n+1}$, then $U$ is an $(n+1)\times (n+1)$ positive semi-definite matrix.
Notationwise $v_i^T v_j$ and $v_i\cdot v_j$ are the same thing in the \emph{Euclidean} space,
so $U=(v_i\cdot v_j)_{1\le i,j\le n+1}$, and all the principal minors of order 2 of $U$
are the squares of the areas of triangles $OP_iP_j$, multiplied by 4.

Denote $\mathbf{1}$ the  $(n+1)$-vector $(1,\dots,1)^T$.
Because the $v_i$'s span the entire $\mathbb{R}^n$ and $\sum v_i=0$,
so $V$ has rank $n$ and $V\cdot\mathbf{1}=0$,
and thus $U=V^TV$ has rank $n$ and contains $\mathbf{1}$ in the null space. 

\begin{definition}
\label{definition_matrix_positive_definite}
Let $\mathcal{F}$ (resp. $\mathcal{F}_0$) be the set of $(n+1)\times (n+1)$ positive semi-definite matrices 
(resp. with rank $n$) with the vector $\mathbf{1}$ in the null space.
\end{definition}

The following result in the reverse direction is also true.

\begin{lemma}
\label{lemma_matrix_simplex}
Let $U\in\mathcal{F}_0$, then $U$ can be treated as the Gram matrix $(v_i\cdot v_j)_{1\le i,j\le n+1}$ 
of some affinely independent vectors $v_1, \dots, v_{n+1}$ of $\mathbb{R}^n$,
and $\sum_{i=1}^{n+1}v_i=0$.
\end{lemma}

The proof is straightforward 
(e.g., we can modify Lemma~\ref{lemma_symmetric_matrix_decomposition} slightly to prove it), so we skip it. 
For any $U\in\mathcal{F}_0$, by Lemma~\ref{lemma_matrix_simplex} we obtain some $v_i$'s,
and by letting $v_i=\overrightarrow{OP_i}$
we obtain an $n$-simplex $P$ in $\mathbb{R}^n$ with $O$ as the centroid.
As $\overrightarrow{P_iP_j}^2$ depends only on $U$ but not the particular realized $v_i$'s,
so there is a one-to-one correspondence between $U\in\mathcal{F}_0$ 
and $n$-simplices in $\mathbb{R}^n$ (with $\sum_{i=1}^{n+1}v_i=0$) up to congruence.
Thus Question~\ref{question_simplex_triangle_rigid} is equivalent to the following question about the matrix $U$.

\begin{question}
\label{question_simplex_matrix_rigid}
Let $U\in\mathcal{F}_0$ for $n\ge 4$.
If a continuous change of $U$ in $\mathcal{F}_0$ preserves all the principal minors of order 2 of $U$,
then does it necessarily preserve $U$ as well?
\end{question}

As mentioned earlier, we do not solve Question~\ref{question_simplex_rigid}, 
which is equivalent to Question~\ref{question_simplex_triangle_rigid} and \ref{question_simplex_matrix_rigid}.
But it seems natural to consider a variant of Question~\ref{question_simplex_matrix_rigid}
where the restriction of positive semi-definite matrices is replaced by a weaker notion of 
\emph{symmetric matrices}. 

\begin{definition}
\label{definition_matrix_symmetric}
Let $\mathcal{U}$ (resp. $\mathcal{U}_0)$ be the set of $(n+1)\times (n+1)$ symmetric matrices 
(resp. with rank $n$) with the vector $\mathbf{1}$ in the null space
and whose diagonal entries and principal minors of order 2 are all non-negative (resp. positive).
\end{definition}

\begin{remark}
It is easy to check that for any $U\in\mathcal{U}$,
if its principal minors of order 2 are all positive, 
then it automatically implies that the diagonal entries are also all positive.
\end{remark}

For this variant of Question~\ref{question_simplex_matrix_rigid}, we construct counterexamples for $n\ge 5$.

\begin{theorem}[Main Theorem 3]
\label{theorem_symmetric_matrix}
For $n\ge 5$, there exists a continuous family of non-identical 
$(n+1)\times (n+1)$ symmetric matrices $U\in\mathcal{U}_0$,
such that all the principal minors of order 2 remain constant.
\end{theorem}

Though Theorem~\ref{theorem_symmetric_matrix} does not fully resolve 
Question~\ref{question_simplex_matrix_rigid}, it is still a somewhat surprising result.
This is because for an $(n+1)\times (n+1)$ symmetric matrix in $\mathcal{U}_0$, 
the degrees of freedom is $\binom{n+1}{2}$
(the same as the number of the principal minors of order 2),
which does not increase the degrees of freedom that we can work on compared 
to the positive semi-definite matrices in $\mathcal{F}_0$.
Our proof of Theorem~\ref{theorem_symmetric_matrix} is constructive, 
and may possibly be used to search for potential counterexamples to
Question~\ref{question_simplex_matrix_rigid} as well,
if they were to exist.

\section{Proof of Theorem~\ref{theorem_symmetric_matrix}}

The following notion will be useful in the proof.

\begin{definition}
Let $\mathcal{D}$ be the set of the $(n+1)\times (n+1)$ symmetric matrices in $\mathcal{U}$
with the vector $\mathbf{1}$ in the null space and whose principal minors of order 2 are all 0.
\end{definition}

\begin{example}
\label{example_symmetric_matrix}
Let $\alpha=(a_1,\dots, a_{n+1})^T$ be a column $(n+1)$-vector that satisfies $\alpha^T\cdot\mathbf{1}=0$,
and $A:=\alpha\alpha^T$ be an $(n+1)\times (n+1)$ symmetric matrix.
Then $A\cdot\mathbf{1}=0$ and $\rank(A)=1$, and $A\in\mathcal{D}$.
\end{example}

Lemma~\ref{lemma_matrix_2_by_2} immediately leads to the following result.

\begin{lemma}
\label{lemma_matrix_non_negative}
If $A, B\in\mathcal{D}$, then for $t>0$, we have $tA+\frac{1}{t}B\in\mathcal{U}$ 
and all the principal minors of order 2 are non-negative constants.
\end{lemma}

Note that unlike $\mathcal{F}$ or $\mathcal{U}$, 
the set $\mathcal{D}$ is not closed under summation of matrices.
To further prove Theorem~\ref{theorem_symmetric_matrix}, 
we only need to improve Lemma~\ref{lemma_matrix_non_negative}
by finding suitable $A, B\in\mathcal{D}$ and $t>0$
such that $tA+\frac{1}{t}B\in\mathcal{U}_0$ (instead of only in $\mathcal{U}$).
This is what we plan to do next.

\begin{remark}
\label{remark_matrix_rank}
As $\mathcal{D}$ is closed under multiplication by a positive factor,
so to prove Theorem~\ref{theorem_symmetric_matrix}, 
in fact we can just ignore $t$ and only search for $A, B\in\mathcal{D}$ 
such that $A+B\in\mathcal{U}_0$.
To make $A+B$ have rank $n$, as $\rank(A)+\rank(B)\ge \rank(A+B)$,
so at least one of the ranks of $A$ and $B$, say $\rank(A)$, should be at least $\lceil n/2 \rceil$,
the least integer greater than or equal to $n/2$.
In Example~\ref{example_symmetric_matrix}, we have $\rank(A)=1$, whose rank is too low for our purpose.
For $n=4$, it can be verified (with the proof skipped and left to the interested reader) 
that any non-zero $5\times 5$ matrix in $\mathcal{D}$ has rank 1, 
less than $\lceil n/2 \rceil=2$, so this method does not work for $n=4$.
\end{remark}

\subsection{An outline of the proof}

If $A\in\mathcal{D}$, then $A$ can be written as $D^TCD$ (not necessarily unique),
where $D=\diag(d_1,\dots, d_{n+1})$ (so $D^T=D$)
and $C$ is an $(n+1)\times (n+1)$ symmetric matrix 
with only $\pm 1$ entries and all the diagonal entries are 1.
As $A$ also satisfies $A\cdot\mathbf{1}=0$, so $D^TCD\cdot\mathbf{1}=0$. 
In Example~\ref{example_symmetric_matrix}, the corresponding $C$ of $A=\alpha\alpha^T$
has all entries as 1.

\begin{remark}
\label{remark_zero_entry}
If all the $d_i$'s in $D$ above have at least \emph{two} zero entries, assume $d_1=d_2=0$.
Then no matter how to choose $B\in\mathcal{D}$, by Lemma~\ref{lemma_matrix_2_by_2}, 
in $A+B$ the determinant of the upper-left $2\times 2$ submatrix (a principal minor of order 2) is 0,
so $A+B$ is \emph{not} in $\mathcal{U}_0$.
So in order to find suitable $A,B\in\mathcal{D}$ to prove Theorem~\ref{theorem_symmetric_matrix},
the $d_i$'s can afford to have at most one zero entry.
\end{remark}

\begin{definition}
Let $\mathcal{H}$ be the finite set of all $(n+1)\times (n+1)$ symmetric matrices 
with only $\pm 1$ entries and all the diagonal entries are 1.
\end{definition}

We remark that unlike $\mathcal{F}$, $\mathcal{U}$ or $\mathcal{D}$, a matrix in $\mathcal{H}$
generally does not contain $\mathbf{1}$ in the null space.
To construct suitable  $A\in\mathcal{D}$, two things to keep in mind:
(1) by Remark~\ref{remark_matrix_rank} we want $\rank(A)$ to be at least $\lceil n/2\rceil$, 
and (2) by Remark~\ref{remark_zero_entry} we want the $d_i$'s to contain at most one zero entry.
Now we follow with some concrete examples to construct $A\in\mathcal{D}$.
The treatment is different for $n$ odd and $n$ even.

\subsection{For $n$ odd}

For an odd $n\ge 5$, let $n=2m-1$. So $n+1=2m$ and $m\ge 3$. 
Let $C_0$ be a $m\times m$ matrix whose diagonal entries are all 1 
and all off-diagonal entries are $-1$. 
As $m\ge 3$, $C_0$ has a full rank of $m$ (but not for $m=2$ where $\rank(C_0)=1$).
Now for a $2m\times 2m$ matrix $C_{2m}$, we cut it into some $2\times 2$ blocks (there are $m\times m$ of them),
fill the $m$ diagonal $2\times 2$ blocks with $\begin{pmatrix} 1 & 1 \\ 1 & 1 \end{pmatrix}$,
and the rest blocks with $\begin{pmatrix} -1 & -1 \\ -1 & -1 \end{pmatrix}$.
As $m\ge 3$ and $\rank(C_0)=m$, so $\rank(C_{2m})=m$.
Thus the null space of $C_{2m}$ has dimension $m$ and the form 
\begin{equation}
\label{equation_null_space_C1}
d=(d_1,d_2,\dots,d_{2m-1},d_{2m})^T=(d_1,-d_1,\dots,d_{2m-1},-d_{2m-1})^T.
\end{equation}
Let $D_{2m}=\diag(d)$ and $A_{2m}=D_{2m}^TC_{2m}D_{2m}$.
As $A_{2m}\cdot\mathbf{1}=D_{2m}^T(C_{2m}d)=0$
and the principal minors of order 2 of $A_{2m}$ are all 0,
then $A_{2m}\in\mathcal{D}$.

If all $d_i$ are non-zero, whose values will be chosen later, then $\rank(A_{2m})=\rank(C_{2m})=m$.
So the null space of $A_{2m}$ has dimension $m$ and the form 
\begin{equation}
\label{equation_null_space_A}
(a_1,a_1,\dots,a_{2m-1},a_{2m-1})^T.
\end{equation}

For a $2m\times 2m$ matrix $C'_{2m}$, let it be obtained 
by moving the last row and last column of $C_{2m}$ to the first row and first column.
So we still have $\rank(C'_{2m})=m$, and the null space of $C'_{2m}$ has the form 
\begin{equation}
\label{equation_null_space_C2}
d'=(d'_1,d'_2,d'_3,\dots,d'_{2m})^T
=(-d'_{2m},d'_2,-d'_2,\dots,d'_{2m})^T.
\end{equation}
Let $D'_{2m}=\diag(d')$ and $B_{2m}={D'_{2m}}^TC'_{2m}D'_{2m}$, then similarly $B_{2m}\in\mathcal{D}$.
If all $d'_i$ are non-zero, whose values will also be chosen later,
then $\rank(B_{2m})=\rank(C'_{2m})=m$, and the null space of $B_{2m}$ has the form 
\begin{equation}
\label{equation_null_space_B}
(b_{2m},b_2,b_2,\dots, b_{2m})^T.
\end{equation}

For any vector in the \emph{common} null space of $A_{2m}$ and $B_{2m}$,
by (\ref{equation_null_space_A}) its $(2k-1)$-th and $2k$-th terms are equal,
and by (\ref{equation_null_space_B}) its $2k$-th and $(2k+1)$-th terms are equal,
so all the terms are equal.
So the common null space of $A_{2m}$ and $B_{2m}$ is 1-dimensional that contains $\mathbf{1}$.
The columns $\alpha_i$ $(1\le i\le 2m)$ of $A_{2m}$ satisfy
\begin{equation}
\label{equation_matrix_A_column_even}
\alpha_1=-\alpha_2, \dots, \alpha_{2m-1}=-\alpha_{2m},
\end{equation}
and the columns $\beta_i$  $(1\le i\le 2m)$ of $B_{2m}$ satisfy
\begin{equation}
\label{equation_matrix_B_column_even}
\beta_2=-\beta_3, \dots, \beta_{2m}=-\beta_1.
\end{equation}
As the common null space of $A_{2m}$ and $B_{2m}$ is 1-dimensional,
so the \emph{row} vectors of $A_{2m}$ and $B_{2m}$ 
span a $(2m-1)$-dimensional subspace.
As $A_{2m}$ and $B_{2m}$ are symmetric matrices,
so by (\ref{equation_matrix_A_column_even}) and (\ref{equation_matrix_B_column_even}),
those $2m$ \emph{column} vectors 
\begin{equation}
\label{equation_column_vectors}
\alpha_1, \dots, \alpha_{2m-1}, \beta_2, \dots, \beta_{2m}
\end{equation}
also span a $(2m-1)$-dimensional subspace.

We next show that for appropriate non-zero $d_i$ and $d'_i$,
the null space of $A_{2m}+B_{2m}$ is also 1-dimensional.

\begin{lemma}
\label{lemma_A_plus_B_U0_even}
For $n=2m-1$ with $m\ge 3$, with proper choice of non-zero $d_i$ and $d'_i$,
then $A_{2m}+B_{2m}\in\mathcal{U}_0$.
\end{lemma}

\begin{proof}
The $2m$ columns of $A_{2m}+B_{2m}$ are $\alpha_i+\beta_i$.
To prove that $A_{2m}+B_{2m}$ has rank $2m-1$ (for proper choice of non-zero $d_i$ and $d'_i$),
we show that the $2m-1$ vectors $\alpha_i+\beta_i$ $(2\le i\le 2m)$ are linearly independent,
which by (\ref{equation_matrix_A_column_even}) is $\beta_i-\alpha_{i-1}$ if $i$ is even,
and by (\ref{equation_matrix_B_column_even}) is $\alpha_i-\beta_{i-1}$ if $i$ is odd and $i\ge 3$.
This is also equivalent to prove that 
$\alpha_1$, $\beta_2$, \dots, $\alpha_{2m-1}$, $\beta_{2m}$ are \emph{affinely independent}.

By (\ref{equation_column_vectors}) the $2m$ vectors 
$\alpha_1$, \dots, $\alpha_{2m-1}$, $\beta_2$, \dots, $\beta_{2m}$
span a $(2m-1)$-dimensional subspace,
so up to a constant factor, there is a \emph{unique} linear dependence among the $2m$ vectors.
For the \emph{particular} case of $d_i=d'_i=(-1)^i$, by symmetry it is not hard to check that
\[\alpha_1+\cdots+\alpha_{2m-1}+\beta_2+\cdots+\beta_{2m}=0.
\]
As the sum of the coefficients is $2m$, not \emph{zero}, so the $2m$ vectors are affinely independent.
By the arguments above, this means that $\alpha_i+\beta_i$ $(2\le i\le 2m)$ are linearly independent.
So for almost all non-zero $d_i$ and $d'_i$, $A_{2m}+B_{2m}$ has rank $2m-1$.

The $2\times 2$ diagonal submatrices of $A_{2m}$ and $B_{2m}$ 
are $\begin{pmatrix} d_i^2 & \pm d_id_j  \\ \pm d_id_j & d_j^2 \end{pmatrix}$
and $\begin{pmatrix} {d'_i}^2 & \pm d'_id'_j  \\ \pm d'_id'_j & {d'_j}^2 \end{pmatrix}$.
Though $d_i$ and $d'_i$ need to satisfy (\ref{equation_null_space_C1}) 
and (\ref{equation_null_space_C2}) respectively,
there are enough freedom for us to further adjust $d_i$ and $d'_i$ such that for any $i\ne j$,
we have $|d_id'_j|\ne |d_jd'_i|$.
So by Lemma~\ref{lemma_matrix_2_by_2}, 
all the principal minors of order 2 of $A_{2m}+B_{2m}$ are positive.
As $A_{2m}+B_{2m}$ has rank $2m-1$ and all the diagonal entries are positive,
so $A_{2m}+B_{2m}\in\mathcal{U}_0$.
\end{proof}

\subsection{For $n$ even}

For an even $n\ge 6$, let $n=2m$. Let $A_{2m}=D_{2m}^TC_{2m}D_{2m}$ be as before
with non-zero $d_i$ in (\ref{equation_null_space_C1}),
and $A_{2m+1}$ be a $(2m+1)\times (2m+1)$ symmetric matrix whose upper-left 
$2m\times 2m$ submatrix is $A_{2m}$ and the other entries are 0. 
Namely, $A_{2m+1}=\begin{pmatrix} A_{2m} & 0  \\ 0 & 0 \end{pmatrix}$.
It is easy to check that $A_{2m+1}\in\mathcal{D}$.
By (\ref{equation_null_space_A}), we have $\rank(A_{2m+1})=m$.
So the null space of $A_{2m+1}$ has dimension $m+1$ and the form
\begin{equation}
\label{equation_null_space_A_2}
(a_1,a_1,\dots,a_{2m-1},a_{2m-1},a_{2m+1})^T.
\end{equation}

Let $B_{2m+1}$ be a $(2m+1)\times (2m+1)$ symmetric matrix whose lower-right 
$2m\times 2m$ submatrix is $A_{2m}$ (\emph{not} $B_{2m}$ as before) and the other entries are 0.
Namely, $B_{2m+1}=\begin{pmatrix} 0 & 0 \\ 0 & A_{2m}  \end{pmatrix}$.
Similarly $B_{2m+1}\in\mathcal{D}$.
So $\rank(B_{2m+1})=m$, and the null space of $B_{2m+1}$ has dimension $m+1$ and the form
\begin{equation}
\label{equation_null_space_B_2}
(b_1,b_2,b_2,\dots,b_{2m},b_{2m})^T.
\end{equation}

For any vector in the common null space of $A_{2m+1}$ and $B_{2m+1}$,
by (\ref{equation_null_space_A_2}) its $(2k-1)$-th and $2k$-th terms are equal,
and by (\ref{equation_null_space_B_2}) its $2k$-th and $(2k+1)$-th terms are equal,
so all the terms are equal.
So the common null space of $A_{2m+1}$ and $B_{2m+1}$ is 1-dimensional that contains $\mathbf{1}$.
The non-zero columns $\alpha_i$ $(1\le i\le 2m)$ of $A_{2m+1}$ satisfy 
\begin{equation}
\label{equation_matrix_A_column_odd}
\alpha_1=-\alpha_2, \dots, \alpha_{2m-1}=-\alpha_{2m},
\end{equation}
and the non-zero columns $\beta_i$ $(2\le i\le 2m+1)$ of $B_{2m+1}$ satisfy 
\begin{equation}
\label{equation_matrix_B_column_odd}
\beta_2=-\beta_3, \dots, \beta_{2m}=-\beta_{2m+1}.
\end{equation}
As the common null space of $A_{2m+1}$ and $B_{2m+1}$ is 1-dimensional,
so the \emph{row} vectors of $A_{2m+1}$ and $B_{2m+1}$ span a $2m$-dimensional subspace.
As $A_{2m+1}$ and $B_{2m+1}$ are symmetric matrices,
so by (\ref{equation_matrix_A_column_odd}) and (\ref{equation_matrix_B_column_odd}),
those $2m$ non-zero \emph{column} vectors 
\begin{equation}
\label{equation_colume_vectors_even}
\alpha_1, \dots, \alpha_{2m-1}, \beta_2, \dots, \beta_{2m}
\end{equation}
also span a $2m$-dimensional subspace, thus are linearly independent.

We next show that the null space of $A_{2m+1}+B_{2m+1}$ is also 1-dimensional.

\begin{lemma}
\label{lemma_A_plus_B_U0_odd}
For $n=2m$ with $m\ge 3$, with proper choice of non-zero $d_i$ $(1\le i\le 2m)$,
then $A_{2m+1}+B_{2m+1}\in\mathcal{U}_0$.
\end{lemma}

\begin{proof}
The $2m+1$ columns of $A_{2m+1}+B_{2m+1}$ are $\alpha_i+\beta_i$, with $\beta_1=\alpha_{2m+1}=0$.
Notice that $\sum_{i=1}^k (\alpha_i+\beta_i)$,
by (\ref{equation_matrix_A_column_odd}) and (\ref{equation_matrix_B_column_odd}) and induction on $k$, 
is $\alpha_k$ if $k$ is odd, and is $\beta_k$ if $k$ is even.
By(\ref{equation_colume_vectors_even})
$\alpha_1$, \dots, $\alpha_{2m-1}$, $\beta_2$, \dots, $\beta_{2m}$ are linearly independent, 
so $\sum_{i=1}^k (\alpha_i+\beta_i)$ $(1\le k\le 2m)$ are linearly independent,
then $\alpha_i+\beta_i$ $(1\le i\le 2m)$ are linearly independent,
and thus $A_{2m+1}+B_{2m+1}$ has rank $2m$.

With proper choice of non-zero $d_i$ $(1\le i\le 2m)$ in (\ref{equation_null_space_C1}),
by Lemma~\ref{lemma_matrix_2_by_2}
(and similar to the proof of Lemma~\ref{lemma_A_plus_B_U0_even}),
all the principal minors of order 2 of  $A_{2m+1}+B_{2m+1}$ are positive.
As  $A_{2m+1}+B_{2m+1}$ has rank $2m$ and all the diagonal entries are positive,
so $A_{2m+1}+B_{2m+1}\in\mathcal{U}_0$.
\end{proof}

The following observation will be useful for the proof of Theorem~\ref{theorem_flexible_volume_non_constant} later.

\begin{remark}
\label{remark_constant_determinant}
For this particular construction of $A_{2m+1}$ and $B_{2m+1}$,
we have $\rank(A_{2m+1})=\rank(B_{2m+1})=m$.
For the matrix $tA_{2m+1}+\frac{1}{t}B_{2m+1}$,
its determinant (and similarly for any of its minors) can be expressed as a sum of signed products 
such that each summand is the product of a minor of $tA_{2m+1}$ (of any order, including 0 and $2m+1$)
and the ``complement'' minor of $\frac{1}{t}B_{2m+1}$.
As any minor of $A_{2m+1}$ and $B_{2m+1}$ of order higher than $m$ is 0,
so the determinant of any $2m\times 2m$ submatrix of $tA_{2m+1}+\frac{1}{t}B_{2m+1}$ is a constant
(because only when $k=m$ the $t^k\times 1/t^{2m-k}$ term is non-vanishing).
This property is analogous to Lemma~\ref{lemma_matrix_2_by_2}.
On the other hand, it is not so for $n=2m-1$. Namely for the matrix $tA_{2m}+\frac{1}{t}B_{2m}$
(not to be confused with the matrix $tA_{2m+1}+\frac{1}{t}B_{2m+1}$),
the determinant of any $(2m-1)\times (2m-1)$ submatrix is not a non-zero constant, in fact, 
it is a linear combination of $t$ and $\frac{1}{t}$ without a constant term.
\end{remark}

\subsection{Summary of the proof of Theorem~\ref{theorem_symmetric_matrix}}

\begingroup
\def\thetheorem{\ref{theorem_symmetric_matrix}}
\begin{theorem}
For $n\ge 5$, there exists a continuous family of non-identical 
$(n+1)\times (n+1)$ symmetric matrices $U\in\mathcal{U}_0$,
such that all the principal minors of order 2 remain constant.
\end{theorem}
\addtocounter{theorem}{-1}
\endgroup

\begin{proof}
By Remark~\ref{remark_matrix_rank},
for $n$ odd it is proved by Lemma~\ref{lemma_A_plus_B_U0_even},
and for $n$ even it is proved by Lemma~\ref{lemma_A_plus_B_U0_odd}.
\end{proof}

\section{Relations to Question~\ref{question_simplex_matrix_rigid}
and Question~\ref{question_simplex_triangle_rigid}}

Recall that $\mathcal{F}_0$ is the set of $(n+1)\times (n+1)$ positive semi-definite matrices 
with rank $n$ and with the vector $\mathbf{1}$ in the null space 
(Definition~\ref{definition_matrix_positive_definite}).
With Theorem~\ref{theorem_symmetric_matrix} proved, 
it remains to see if the \emph{positive semi-definiteness} of the matrices in $\mathcal{F}_0$ 
(as compared to the symmetric matrices in $\mathcal{U}_0$)
is a barrier to solve Question~\ref{question_simplex_matrix_rigid},
or our method to prove Theorem~\ref{theorem_symmetric_matrix} 
can also be adapted to provide counterexamples to Question~\ref{question_simplex_matrix_rigid},
namely, if we can find $A,B\in\mathcal{D}$ such that $A+B\in\mathcal{F}_0$ 
(see Remark~\ref{remark_matrix_rank}).

\begin{remark}
If such $A,B\in\mathcal{D}$ were to exist such that $A+B\in\mathcal{F}_0$,
because $\mathcal{F}_0$ in an \emph{open} subset of $\mathcal{U}$ (Definition~\ref{definition_matrix_symmetric}),
so for any small changes of $A$ and $B$ in $\mathcal{D}$ their sum $A+B$ is still in $\mathcal{F}_0$.
This also makes numerical search possible, at least in theory:
while $\mathcal{D}$ is a high (up to $n$) dimensional space, it is fairly easy to find all the elements of $\mathcal{D}$,
so we can just sample $A, B\in\mathcal{D}$ numerically, and then check if $A+B\in\mathcal{F}_0$. 
If the numerical search turns out successful for some $n$, then it provides counterexamples to 
Question~\ref{question_simplex_matrix_rigid}, as well as to 
Question~\ref{question_simplex_triangle_rigid} and Question~\ref{question_simplex_rigid}
as they are all equivalent.
\end{remark}

In general, no matter the numerical search is successful or not,
in the following we provide counterexamples to a variant of Question~\ref{question_simplex_triangle_rigid}
where $\mathbb{R}^n$ is replaced by a pseudo-Euclidean space $\mathbb{R}^{p,n-p}$ 
for some unspecified $p\ge 2$.

\begin{theorem}
\label{theorem_simplex_triangle_non_rigid}
For any $n\ge 5$, in $\mathbb{R}^{p,n-p}$ for some unspecified $p\ge 2$,
there exists a continuous family of non-congruent $n$-simplices $P$
with the centroid fixed at the origin $O$, such that all triangles $OP_iP_j$ have fixed areas
and are in Euclidean planes.
\end{theorem}

\begin{proof}
By Lemma~\ref{lemma_A_plus_B_U0_even} for $n$ odd 
and Lemma~\ref{lemma_A_plus_B_U0_odd} for $n$ even,
we can find $(n+1)\times (n+1)$ symmetric matrices $A,B\in\mathcal{D}$ 
such that $A+B\in\mathcal{U}_0$.
Denote $A+B$ by $U$, the upper-left $n\times n$ submatrix by $U_0$,
the upper $n\times (n+1)$ submatrix by $U_1$,
and the left $(n+1)\times n$ submatrix by $U_2$.
As $U\cdot\mathbf{1}=0$ and $U_2=U_1^T$, so
$\rank(U_0)=\rank(U_1)=\rank(U_2)=\rank(U)$.
As $U\in\mathcal{U}_0$, so $\rank(U)=n$ and $\rank(U_0)=n$, 
and therefore $\det(U_0)$ is non-zero.
As $U_0$ is symmetric, by Lemma~\ref{lemma_symmetric_matrix_decomposition},
$U_0$ may be treated as the Gram matrix $(v_i\cdot v_j)_{1\le i,j\le n}$
of some linearly independent vectors $v_1, \dots, v_n$ of a pseudo-Euclidean space 
$\mathbb{R}^{p,n-p}$ for some $p$.
Let $v_{n+1}=-\sum_{i=1}^n v_i$, then $v_1, \dots, v_{n+1}$ are affinely independent,
and by applying $U\cdot\mathbf{1}=0$ again,
we have $U=(v_i\cdot v_j)_{1\le i,j\le n+1}$.

In $\mathbb{R}^{p,n-p}$, let $P$ be an $n$-simplex whose vertices $P_i$ $(1\le i\le n+1)$ 
satisfy $v_i=\overrightarrow{OP_i}$.
As $\sum_{i=1}^{n+1} v_i=0$, then $O$ is the centroid of $P$.
As $U\in\mathcal{U}_0$, so all the diagonal entries and principal minors of order 2 are all positive.
Therefore any triangle $OP_iP_j$ spans a two-dimensional Euclidean plane, and thus $p\ge 2$.

For any $t$ in a small neighborhood of 1, we still have $tA+\frac{1}{t}B\in\mathcal{U}_0$.
So we can construct a continuous family of non-congruent simplices $P$ 
in $\mathbb{R}^{p,n-p}$ from $tA+\frac{1}{t}B$.
By Lemma~\ref{lemma_matrix_non_negative}, all triangles $OP_iP_j$ have fixed areas over $t$.
This finishes the proof.
\end{proof}

\begin{remark}
In Theorem~\ref{theorem_simplex_triangle_non_rigid}, for a fixed $n$, the $p$ may not be uniquely determined.
Even when $A$ and $B$ are fixed, $tA+\frac{1}{t}B$ and $A+B$ 
may potentially correspond to different $p$'s when $t$ moves away from 1,
where $p$ is the number of positive eigenvalues of $tA+\frac{1}{t}B$
(Lemma~\ref{lemma_symmetric_matrix_decomposition}).
\end{remark}

We have the following interesting observation,
which will be further addressed in the next section.

\begin{remark}
\label{remark_constant_volume_P}
For $n=2m$, in Lemma~\ref{lemma_A_plus_B_U0_odd},
for the \emph{particular} construction of $A_{2m+1}$ and $B_{2m+1}$,
we have $\rank(A_{2m+1})=\rank(B_{2m+1})=m$.
By constructing a continuous family of non-congruent $2m$-simplices $P$ 
from $tA_{2m+1}+\frac{1}{t}B_{2m+1}$ using Theorem~\ref{theorem_simplex_triangle_non_rigid},
and using the formula (\ref{equation_volume_pseudo_Euclidean}) to compute the volumes
of the $2m$-simplices formed by $O$ and the $2m+1$ facets of $P$,
then by Remark~\ref{remark_constant_determinant}, 
we show that the volumes of those $2m$-simplices remain constant over $t$.
As those simplices add up to $P$, so $V_{2m}(P)$ also remains constant over $t$ 
(but not so for $n=2m-1$).
\end{remark}

We use Theorem~\ref{theorem_simplex_triangle_non_rigid} 
to prove Theorem~\ref{theorem_flexible_pseudo_Euclidean} next.

\section{Proofs of Theorem~\ref{theorem_flexible_pseudo_Euclidean}
and \ref{theorem_flexible_volume_non_constant}}
\label{section_dual_pseudo_Euclidean}

In $\mathbb{R}^{p,n-p}$, let $P$ be an $n$-simplex with the centroid at the origin $O$
(but for now the triangles $OP_iP_j$ need not all be in Euclidean planes 
as required in Theorem~\ref{theorem_simplex_triangle_non_rigid}).
For some $c>0$, let $Q$ be a \emph{dual} of $P$ 
\[Q:=\{y \in \mathbb{R}^{p,n-p}: x\cdot y \le c \quad\text{for all $x\in P$}\},
\]
see also (\ref{equation_dual_Euclidean}).
When $c=1$, $Q$ is the \emph{polar dual} $P^{\ast}$ of $P$.
Let $v_i=\overrightarrow{OP_i}$. Just like in the Euclidean case, 
it can be verified that $Q$ is an $n$-simplex in a finite region with the centroid at the origin $O$,
and $v_i$ is a normal vector at the $(n-1)$-face $F_i$ of $Q$. 
But unlike in $\mathbb{R}^n$, here $v_i$ is an \emph{outward} normal vector to $Q$ at $F_i$ if $v_i^2>0$,
an \emph{inward} normal vector to $Q$ at $F_i$ if $v_i^2<0$, or \emph{parallel} to $F_i$ if $v_i^2=0$.
However, this does not affect our results.

Now for $P$ also let all triangles $OP_iP_j$ be in Euclidean planes 
(as required in Theorem~\ref{theorem_simplex_triangle_non_rigid}).
This insures that the area of $OP_iP_j$ is non-zero.
In this case of $P$, we have $v_i^2>0$ for all $i$, 
so $v_i$'s are still all outward normal vectors to $Q$ at $F_i$'s.
In $\mathbb{R}^n$, 
the volume of any face of $Q$ (of any dimension) can be computed from the information of $P$
using a simple formula (e.g. see \cite[Theorem 19]{Lee:stress});
the formula only uses the fact that $O$ is the centroid of $P$
and can be applied to $\mathbb{R}^{p,n-p}$ as well.
As a direct result, it shows that 
for the $(n-2)$-face $F_{ij}$ of $Q$ on the intersection of $F_i$ and $F_j$,
the volume $V_{n-2}(F_{ij})$ is proportional to the area of the triangle $OP_iP_j$
\begin{equation}
\label{simplex_dual_volume_n_minus_2}
V_{n-2}(F_{ij})=\frac{c_2V_2(OP_iP_j)}{V_n(P)}\cdot c^{n-2},
\end{equation}
where $c_2$ is a constant that only depends on $n$.
Similarly we also have
\begin{equation}
\label{simplex_dual_volume_n}
V_n(Q)=\frac{c_0}{V_n(P)}\cdot c^n,
\end{equation}
where $c_0$ is also a constant that only depends on $n$.


Now we are ready to prove Theorem~\ref{theorem_flexible_pseudo_Euclidean},
which we recall below.

\begingroup
\def\thetheorem{\ref{theorem_flexible_pseudo_Euclidean}}
\begin{theorem}
For any $n\ge 5$, in $\mathbb{R}^{p,n-p}$ for some unspecified $p\ge 2$,
there exists a continuous family of non-congruent $n$-simplices $Q$
with fixed $(n-2)$-volumes of all the $(n-2)$-faces,
and all the dihedral angles are Euclidean angles.
\end{theorem}
\addtocounter{theorem}{-1}
\endgroup

\begin{proof}
By Theorem~\ref{theorem_simplex_triangle_non_rigid},
there exists a continuous family of non-congruent $n$-simplices $P$ 
with the centroid fixed at the origin $O$,
such that all triangles $OP_iP_j$ have fixed areas and are in Euclidean planes.
Let $Q$ be a dual of $P$ for some factor $c$ ($c$ may vary when $P$ varies),
then the dihedral angles of $Q$ are all Euclidean angles.
With the proper scale of $c$ for $Q$, by (\ref{simplex_dual_volume_n_minus_2}),
the $(n-2)$-volumes of all the $(n-2)$-faces of $Q$ can be adjusted to remain constant during the deformation.
This finishes the proof.
\end{proof}

Now with both Theorem~\ref{theorem_flexible_pseudo_Euclidean_n_4}
and \ref{theorem_flexible_pseudo_Euclidean} proved,
for any $n$-simplex $Q$ in $\mathbb{R}^{p,n-p}$ ($n\ge 4$)
with fixed volumes of codimension 2 faces during a deformation,
in the spirit of the bellows conjecture \cite{Sabitov:invariance,ConnellySabitovWalz}, or for other heuristics,
it seems natural to ask the following question.

\begin{question}
\label{question_flexible_simplex_volume}
For $n\ge 4$ in $\mathbb{R}^{p,n-p}$ for some $p$, 
if there exists a continuous family of $n$-simplices $Q$ 
with fixed $(n-2)$-volumes of all the $(n-2)$-faces,
then does the volume of $Q$ also necessarily remain constant?
\end{question}

For $n=4$ in $\mathbb{R}^{3,1}$, we answer this question negatively in
Corollary~\ref{corollary_flexible_volume_non_constant_4}.
For $n\ge 5$, we give a loose discussion on this topic based on the counterexamples
we constructed in Theorem~\ref{theorem_flexible_pseudo_Euclidean}.
By (\ref{simplex_dual_volume_n_minus_2}) and (\ref{simplex_dual_volume_n}),
because both $V_{n-2}(F_{ij})$ and $V_2(OP_iP_j)$ remain constant during the deformation,
we can show that for the volume of $Q$ to remain constant,
\emph{if and only if} both $V_n(P)$ and $c$ also remains constant during the deformation;
and if so, we can set $c=1$ and let $Q$ be the \emph{polar} dual $P^{\ast}$ of $P$.
Regarding the volume of $P$, it is addressed in Remark~\ref{remark_constant_volume_P}
that for the \emph{particular} counterexamples of $P$ we constructed in $\mathbb{R}^{p,n-p}$
in Theorem~\ref{theorem_simplex_triangle_non_rigid} (from Lemma~\ref{lemma_A_plus_B_U0_odd}),
the volume of $P$ \emph{does} remain constant for $n$ even but not for $n$ odd.
Thus, for the particular counterexamples of $Q$ we constructed in $\mathbb{R}^{p,n-p}$
in Theorem~\ref{theorem_flexible_pseudo_Euclidean},
the volume of $Q$ remains constant for $n$ even but not for $n$ odd.
This answers Question~\ref{question_flexible_simplex_volume} negatively if $n$ is odd,
and combining with Corollary~\ref{corollary_flexible_volume_non_constant_4} for $n=4$,
we prove Theorem~\ref{theorem_flexible_volume_non_constant}.

\begingroup
\def\thetheorem{\ref{theorem_flexible_volume_non_constant}}
\begin{theorem}
If $n=4$ or $n\ge 5$ and $n$ is odd, in $\mathbb{R}^{p,n-p}$ for some unspecified $p\ge 2$,
there exists a continuous family of non-congruent $n$-simplices $Q$
with fixed $(n-2)$-volumes of all the $(n-2)$-faces,
but the volume of $Q$ does not remain constant.
\end{theorem}
\addtocounter{theorem}{-1}
\endgroup

While Question~\ref{question_flexible_simplex_volume} is still open if $n$ is even,
our construction provides some positive evidence for it.





\bibliographystyle{abbrv}  
\bibliography{codimension2_arxiv}   

\begin{thebibliography}{10}

\bibitem{Alexandrov:flexible}
V.~Alexandrov.
\newblock An example of a flexible polyhedron with nonconstant volume in the
  spherical space.
\newblock {\em Beitr. Algebra Geom.}, 38(1):11--18, 1997.

\bibitem{Alexandrov:Minkowski}
V.~Alexandrov.
\newblock Flexible polyhedra in {Minkowski} 3-space.
\newblock {\em Manuscr. Math.}, 111(3):341--356, 2003.

\bibitem{Connelly:counterexample}
R.~Connelly.
\newblock A counterexample to the rigidity conjecture for polyhedra.
\newblock {\em Publ. Math. Inst. Hautes \'{E}tud. Sci.}, 47:333--338, 1977.

\bibitem{ConnellySabitovWalz}
R.~Connelly, I.~Sabitov, and A.~Walz.
\newblock The bellows conjecture.
\newblock {\em Beitr. Algebra Geom.}, 38(1):1--10, 1997.

\bibitem{Gaifullin:volume}
A.~A. Gaifullin.
\newblock Volume of a simplex as a multivalued algebraic function of the areas
  of its two-faces.
\newblock In {\em Topology, geometry, integrable systems, and mathematical
  physics}, volume 234 of {\em Amer. Math. Soc. Transl. Ser. 2}, pages
  201--221. Amer. Math. Soc., Providence, RI, 2014.

\bibitem{Lee:stress}
C.~W. Lee.
\newblock P.{L}.-spheres, convex polytopes, and stress.
\newblock {\em Discrete Comput. Geom.}, 15(4):389--421, 1996.

\bibitem{McMullen:simplices}
P.~McMullen.
\newblock Simplices with equiareal faces.
\newblock {\em Discrete Comput. Geom.}, 24(2--3):397--411, 2000.

\bibitem{MoharRivin}
B.~Mohar and I.~Rivin.
\newblock Simplices and spectra of graphs.
\newblock {\em Discrete Comput. Geom.}, 43(3):516--521, 2010.

\bibitem{Sabitov:invariance}
I.~Sabitov.
\newblock On the problem of invariance of the volume of a flexible polyhedron.
\newblock {\em Russ. Math. Surv.}, 50(2):451--452, 1995.

\bibitem{Sabitov:algebraic}
I.~K. Sabitov.
\newblock Algebraic methods for solution of polyhedra.
\newblock {\em Russ. Math. Surv.}, 66(3):445--505, 2011.

\bibitem{Zhang:rigidity}
L.~Zhang.
\newblock Rigidity and volume preserving deformation on degenerate simplices.
\newblock {\em Discrete Comput. Geom.}, 60(4):909--937, 2018.

\bibitem{Zhang:lifting}
L.~Zhang.
\newblock Lifting degenerate simplices with a single volume constraint.
\newblock {\em Beitr. Algebra Geom.}, 61(2):335--353, 2020.

\end{thebibliography}

%
%

\end{document}